\documentclass[11pt,a4paper]{article}

\usepackage[a4paper,top=1.2in, bottom=1.2in, left=1.2in, right=1.2in]{geometry}
\usepackage{commath}
\usepackage{amsmath,paralist,amsthm,mathpazo}
\usepackage{amsfonts,amssymb}
\usepackage{lmodern}
\usepackage{ulem}
\usepackage{setspace}
\usepackage{enumerate}
\usepackage{pgfplots}
\usepackage{float}
\usepackage{enumitem}
\usepackage{tikz}
\usepackage{multirow,array}
\usepackage{pbox}
\usepackage{graphicx}
\usepackage[utf8]{inputenc}
\usepackage{amssymb}
\usepackage{adjustbox}
\usepackage{fancyref}
\usepackage[utf8]{inputenc}
\usepackage{scalerel}
\usepackage[round]{natbib}
\usepackage{lineno}
\usepackage{setspace}
\graphicspath{ {E:/images/} }

\usepackage[justification=centering]{caption}
\usetikzlibrary{positioning}

\newtheorem{theorem}{Theorem}
\newtheorem{proposition}{Proposition}

\newtheorem{lemma}{Lemma}

\newtheorem*{remark}{\textit{Remark}}
\newtheorem{assumption}{Assumption}

\newtheoremstyle{as}
{\topsep}
{\topsep}
{\itshape}
{}
{\scshape}
{.}
{ }
{\thmname{\textbf{\textit{#1}}}\thmnumber{ $\mathbf{(#2)}$}\thmnote{ (#3)}}       
\theoremstyle{as}

\newtheoremstyle{pro}
{\topsep}
{\topsep}
{\itshape}
{}
{\scshape}
{.}
{ }
{\thmname{\textbf{\textit{#1}}}\thmnumber{ $\mathbf{#2}$}\thmnote{ (#3)}}%
\theoremstyle{pro}

\newtheoremstyle{lem}
{\topsep}
{\topsep}
{\itshape}
{}
{\scshape}
{.}
{ }
{\thmname{\textbf{\textit{#1}}}\thmnumber{ $\mathbf{#2}$}\thmnote{ (#3)}}%
\theoremstyle{lem}

\usepackage{subfigure}
\usepackage{tikz}

\begin{document}
\centerline{\textbf{\large Evolutionary, Mean-Field and Pressure-Resistance}}
\centerline{\textbf{\large Game Modelling of Networks Security}}
\bigskip
\centerline{\textbf{Stamatios Katsikas}\footnote{\scriptsize Centre for Complexity Science, University of Warwick, Coventry, CV4 7AL, UK,  s.katsikas@warwick.ac.uk},\textbf{Vassilli Kolokoltsov}\footnote{\scriptsize Department of Statistics, University of Warwick, Coventry, CV4 7AL, UK, associate member of Informatics Problems of the Federal Research Center "Computer Science and Control" of RAS, v.kolokoltsov@warwick.ac.uk}}
\smallskip
\centerline{\today}
\bigskip
\begin{abstract}
\noindent  
The recently developed mean-field game models of corruption and bot-net defence in cyber-security, the evolutionary game approach to inspection and corruption, and the pressure-resistance game element, can be combined under an extended model of interaction of large number of indistinguishable small players against a major player, with focus on the study of security and crime prevention. In this paper we introduce such a general framework for complex interaction in network structures of many players, that incorporates individual decision making inside the environment (the mean-field game component), binary interaction (the evolutionary game component), and the interference of a principal player (the pressure-resistance game component). To perform concrete calculations with this overall complicated model, we suggest working, in sequence, in three basic asymptotic regimes; fast execution of personal decisions, small rates of binary interactions, and small payoff discounting in time.
\end{abstract}

\section{Introduction}\label{sec3.1}

The issue of social security and crime prevention dominantly concerns the modern societies. In the traditional terrain of counter-terrorism, corruption and tax evasion, the corresponding authorities in charge struggle to deal with large populations of increasingly informed violating individuals (this term will be used interchangeably with the terms agents or small players). Reversely, in the recently emerging field of cyber-security, large groups of individuals aim to defend their private computers against a lurking cyber-criminal (bot-net defence). Similar reasoning can be asserted for the citizens of a large city defending against a biological weapon (bio-terrorism). The rapid advance in means of interaction, communication and exchange of information has established the individuals' social network as a decisive parameter of their strategies in the above and similar instances. Here we consider agents who are organized in specific social or phenotypic (or even geographical), and behavioural network structures. The central focus of this paper is to investigate the evolution of the complex process where a (very) large number of interacting individuals, susceptible to engage in or be affected by criminal behaviour, decide their strategies subject to a benevolent, or respectively to a malicious, major player's (this term will be used interchangeably with the term principal) pressure, to their individual optimization criterion, and to their (social) environment's influence.

In the real life scenaria we aim to capture with our approach, it is natural to distinguish two main dimensions of structure. The first dimension refers to the individuals' objective distribution among different levels of social, bureaucratic, or phenotypic hierarchy, or to their geographical distribution, in general to any finite partition according to their independent characteristics. One can think for example of tax payers of different bands, employees of different grades, or infected computers/individuals of different degrees. The second dimension refers to the agents' distribution among different types of strategy or behaviour, subject (mainly) to the agents' individual control; say for example the level of tax evasion in the field of inspection games, the extent of bribery acceptance in the field of corruption games, or the level of defence against terrorist activity or a malware  in the fields of counter terrorism and cyber-security respectively. 

Note that our game theoretic approach is developed under the  basic idea of a very large number of non-cooperative, interacting agents playing against (i.e. under the pressure of) a single major player. In principle, our model belongs to the wide class of non-linear Markov games, see e.g., \cite{kol1}, combining under an extended scheme the pressure-resistance, the evolutionary, and the mean-field game approach. 

The pressure-resistance terminology was introduced in \cite{Kol-3}, where ideas captured from evolutionary game theory were extended, including the pressure of a principal player on a large group of interacting agents. Here, the pressure-resistance game component refers to the principal's interference generatin transitions solely on the first dimension of structure (e.g. a benevolent director being able to promote or downgrade interacting bureaucrats, computers and individuals getting infected or recovering subject to a cyber-criminal's and a bio-terrorist's activity respectively). This approach of major and minor players has also been considered for the analysis of mean-field type models, see, e.g., \cite{hu,ben,carm}. 

The evolutionary game component refers to the agents' pairwise interactions, with a particular focus on the established social norms effect, potentially generating transitions on both dimensions of structure. For a general survey on the literature of population dynamics applications on game theory, that is, on evolutionary game theory, see, e.g., \cite{sm,weib,gint,sam,hob,tay,sza}, and references therein. See also \cite{fried1,fried2} for specific application in economics.

The mean-field game component refers to the agents' individual optimization controlled by their strategic position on the second dimension of structure, taking into account the entire population's behaviour. This element of `globally' rational optimization introduces an additional level of complexity compared to \cite{kats,Kol-3}, where optimization strictly upon imitation of successful strategies on the basis of binary comparison of payoffs was considered (purely evolutionary approach). Mean-field games (MFG) were introduced by \cite{lar}, by analogy with the mean-field theory in statistical mechanics, and were also developed independently by \cite{huma} as large population stochastic dynamic games. In principle, they represent a natural extension of earlier work in the economics literature under the assumption of infinite number of players, see, e.g., \cite{Au,dub} for static games, \cite{jov,ber} for dynamic games. The literature on mean-field games is growing really fast, see, e.g., \cite{tem,card,cain,ben0,carm0,gom,bau2}, and references therein for a general survey.

Here we shall work in three asymptotic regimes; fast execution of the agents' personal decisions, weak binary interactions, and small discounting in time. The need to introduce this ternary asymptotic approach is revealed from the analysis of a similar setting conducted by \cite{kol5}, where the distribution of infection in a computers network with a malicious software controlled by a cyber-criminal was described by a stationary MFG model with four states. Whilst the three states model describing the distribution of corruption in a population of bureaucrats under the pressure of a benevolent principal that was studied by \cite{Kol-4}, is solved explicitly without any asymptotic simplifications, the introduction of a fourth state in \cite{kol5} already increases the complexity significantly, such that the need to consider (though not as strongly as we do here) the assumption of large $\lambda$ (fast decisions execution) is critical to obtain descent solutions.

Similarly, for the more complicated $n\times m$ states model we introduce here, the need to consider the three asymptotic regimes mentioned above becomes obvious. In principle, even without working in these regimes one can sometimes obtain explicit but extremely lengthy formulas, not revealing any clearer insights. But also form a practical point of view our asymptotic approach has clear interpretation. Indicatively, an infinitely large transition rate $\lambda$ implies the natural process of immediate execution of personal decisions as long as they have already been taken, while a vanishingly small discounting coefficient $\delta_{dis}$ implies a short planning horizon. Both of the models studied in \cite{kol5,Kol-4}, as well as the extended approach presented in this paper, belong to the category of finite state space mean-field games that were initially considered by \cite{gom1,gom2}. See also \cite{gom3} specifically for socio-economic applications.

In addition to the above applications on corruption, and cyber-security, here we introduce the bio-terrorism interpretation, that is, the defence of a population against a biological weapon. The implementation of game theoretic methods to the analysis of terrorism has been vastly developed ever since the 1980s, with game theory allowing the investigation of different strategic interactions (e.g. terrorists vs government, terrorists vs terrorists, terrorists for sponsors, terrorists for supporters), see, e.g., \cite{san1,san2,san3} and references therein. The additional pairing captured here is civilians vs a bio-terrorist.

We organize the paper as follows. In Section \ref{sec3.2} we introduce our model, specifying explicitly the time-dependent and stationary MFG consistency problems. In Section \ref{sec3.3} we solve the stationary problem in our proposed asymptotic regimes, and we show that the identified solution is a stable fixed point of the corresponding evolutionary dynamics. In Section \ref{sec3.4} we obtain our main result; we construct the class of time-dependent solutions that stay in a neighbourhood of the identified stationary solution. In the terminology of mathematical economics, this stationary solution represents a turnpike (see, e.g., \cite{Kol-1,zas}) for the class of time-dependent solutions. In Section \ref{sec3.5} we summarize our approach and our results.

\section{Formal Model}\label{sec3.2}

Let $H=\{1, \dots , |H|=n\}$ be a finite set characterizing the hierarchical partition of small players inside the environment, say their position in the bureaucratic staircase of an organization. Alternatively, it may describe the extend of individuals' infection to a bio-weapon. Moreover, let $B=\{1, \dots , |B|=m\}$ be a finite set characterizing the behavioural or strategic partition of agents, say the level of compliance with official regulations, the extend of involvement in terrorist activity, or the degree of protection for PCs/citizens against cyber-criminals/bio-terrorists. Then, the states of an agent are given by ordered pairs of the form $(h,b)$, with $h\in H$, $b\in B$, the finite state space being $S=H\times B$.

\begin{remark}\label{re3.2}
In some cases it is reasonable to include an additional zero state, some kind of a rank-less sink, where no choice of $B$ is available, say for example a corrupted civil servant suspended from duty without the potential to be bribed, an infected individual put in quarantine, and so forth. Thus, the state space can be either $S=H\times B$ as initially defined above, or $\tilde S=H\times B \cup \{0\}=S\cup \{0\}$ as alternatively implied with this comment. We shall stick here to the first instance.
\end{remark}

\noindent We distinguish the following three structures.

Firstly, the \textit{decision structure} $(B,E_D,\lambda )$, that is a non-oriented graph with the set of vertices $B$ and the set of edges $E_D$, where an edge $e$ joins the vertices $i$ and $j$ whenever an agent is able to switch between states $(h,i)$ and $(h,j)$. Every such transition in $B$ requires certain random $\lambda$-exponential time. For simplicity, a single parameter $\lambda$ is chosen for all possible transitions. As mentioned, we shall mostly look at the asymptotic regime with $\lambda \rightarrow \infty$. We take the agents to be homogeneous and indistinguishable, in the sense that their strategies and payoffs may depend only on their states, and not on any other individual characteristics. Hence, a decision of an agent in a state $(h,b)$ at any time is given by the \textit{decision matrix} $u=(u_{h b\to h\tilde b})$, expressing his intention to switch from strategy $b$ to $\tilde b$, for all $\tilde b\in B$ such that $\tilde b\neq b$. In our model, we consider agents without mixed strategies, that is, for any state $(h,b)$ the decision vector $(u_{h b\to h\tilde b})$ is either identically zero, when the agent does not wish to change strategy, or there exists one strategy $b_1 \neq b$ such that $u_{hb\to hb_1}=1$, and all the other coordinates of $(u_{h b\to h\tilde b})$ being zero, when the agent wishes to change from strategy $b$ to $b_1$.

Secondly, the \textit{pressure structure} $(H, E_P, q_{jb\to ib})$, that is an oriented graph, where an edge $e$ joins the vertices $j$ and $i$ whenever a major player has the power to upgrade or downgrade the small players from the hierarchy level $j$ to $i$. In this case, coefficients $q_{jb\to ib}$ represent the rates of such transitions in $H$, that is, every such transition requires certain $q_{jb\to ib}$-exponential waiting time. In general, these rates may depend on some principal's control (one can think of some parameter describing his/her efforts, for example his/her budget). We shall not exploit this version here.

Finally, we consider the \textit{evolution structure} that characterizes the change in the distribution of states due to the agents' pairwise interaction (e.g. through exchange of opinions, fight with competitors, effect of established social norms etc.). This can be described by the set of rates $q^s_{s_1\to s_2}/N$, by which an agent in state $s$ can stimulate the transition of another agent from state $s_1$ to state $s_2$. For instance, an honest agent (or even a corrupted one) may help the principal to discover, and therefore punish, the illegal behaviour of a corrupted agent. Note that transitions due to binary interaction can be naturally separated into transitions in $B$ and transitions in $H$, yielding respectively the \textit{behavioural} and the \textit{hierarchical evolution structures}.

\begin{remark}
The scaling $1/N$ for the rates of binary interactions is the standard procedure of making the strength of $N^2$ (number of pairs) binary transitions comparable to the strength of $N$ unilateral transitions. 
\end{remark}

Here we shall ignore the behavioural element of the evolution structure. That is, we shall assume that transition rates $q^s_{s_1\to s_2}/N$ may not vanish only for two states $s_1$, $s_2$ that differ strictly in their $h$-component. Moreover, since we shall work in the asymptotic regime of small binary interactions, it would be helpful to introduce directly a small parameter $\delta_{int}$ discounting the power of these interactions. Then, we shall denote thereafter the corresponding transition rates by $\delta_{int}\cdot q_{h_{1}b\to h_{2}b}^{s}/N$.

\begin{remark}\label{re3.3}
The evolutionary transitions in $B$ represent an alternative to the individual transitions described by the decision structure $(B,E_D,\lambda)$, and can be considered negligible in the limit $\lambda \to \infty$ that we shall look at here. Taking into account a behavioural evolution structure is more appropriate in the absence of a decision structure, which was the case developed by \cite{Kol-4}.
\end{remark}

To introduce a more detailed description of our game theoretic framework, note that the states of the corresponding $N$ players game are the $N$-tuples of the~form, 

$$\{(h_{1},b_{1}),\dots,(h_{N},b_{N})\},$$ 

\noindent where each pair $(h_{i},b_{i})$ describes each of the $N$ players position on the hierarchy and the behaviour axis respectively. Assuming that each player adopts a decision matrix $u$, then the system evolves according to the continuous time Markov chain introduced above, with the corresponding transitions rates as were specified. If we further specify the rewards for staying in each state per unit of time, the transition fees (or costs) for transiting form one state to another, as well as the terminal payoffs corresponding to each state for some finite terminal time, then we shall be working in the setting of a stochastic dynamic game of finite number of players.

As usual in a mean-field game approach, we are interested in estimating the approximate symmetric Nash equilibria. Assuming indistinguishable agents, the system's state space can be reduced to the set $Z^{nm}_{+}$ of vectors $n=(n_{ij})$, $i\in H$, $j\in B$, where $n_{ij}$ denotes the number of agents in state $(i,j)$, and $N=\sum_{ij}n_{ij}$ denotes the (constant) total number of agents. 

Therefore, the initially introduced Markov chain reduces to the Markov chain on $Z^{nm}_{+}$, described by the time-dependent generator:

\begin{align}
\begin{gathered}\label{eq3.1}
L_{N}^{t}F(n)=\sum_{a}^{n}\sum_{\beta}^{m}\sum_{c}^{n}n_{a\beta}\cdot q_{a\beta\to c\beta}\cdot \bigl(F(n_{a\beta}^{c\beta})-F(n)\bigr)\\
+\sum_{a}^{n}\sum_{\beta}^{m}\sum_{c}^{n}\sum_{\gamma}^{n}\sum_{k}^{m}n_{a\beta}\cdot \delta_{int}\cdot q_{a\beta\to c\beta}^{\gamma k}/N\cdot n_{\gamma k}\cdot \bigl(F(n_{a\beta}^{c\beta})-F(n)\bigr)\\
+\sum_{a}^{n}\sum_{\beta}^{m}\sum_{c}^{m}n_{a\beta}\cdot \lambda\cdot u_{a\beta\to ac}\cdot \bigl(F(n_{a\beta}^{ac})-F(n)\bigr),
\end{gathered}
\end{align}

\noindent where the unchanged values in the arguments of function $F$ on the right-hand side are omitted. Equivalently, in the normalized version the system's state space can be reduced to the subset of the probability simplex $\Sigma_{n\times m}^{N}\subseteq\mathbb{R}^{n\times m}$, with vectors of the form $x=(x_{ij})=n/N$, $i\in H$, $j\in B$, where each coordinate represents the occupation density (alternatively the occupation probability) of each state $(i,j)$. 

For the Markov chain on $\Sigma_{n\times m}^{N}$, generator \eqref{eq3.1} can be rewritten in the equivalent form:

\begin{align}
\begin{gathered}\label{eq3.2}
L_{N}^{t}f(x)=\sum_{a}^{n}\sum_{\beta}^{m}\sum_{c}^{n}x_{a\beta}\cdot N\cdot q_{a\beta\to c\beta}\cdot\bigl(f(x+(e_{c\beta}-e_{a\beta})/N)-f(x)\bigr)\\
+\sum_{a,c,\gamma}^{n}\sum_{\beta,k}^{m}x_{a\beta}\cdot N\cdot \delta_{int}\cdot q_{a\beta\to c\beta}^{\gamma k}/N\cdot x_{\gamma k}\cdot N\cdot \bigl(f(x+(e_{c\beta}-e_{a\beta})/N)-f(x)\bigr)\\
+\sum_{a}^{n}\sum_{\beta}^{m}\sum_{c}^{m}x_{a\beta}\cdot N\cdot \lambda\cdot u_{a\beta\to ac}\cdot \bigl(f(x+(e_{ac}-e_{a\beta})/N)-f(x)\bigr),
\end{gathered}
\end{align}

\noindent where $\{e_{ij}\}$ is the standard orthonormal basis in $\mathbb{R}^{n\times m}$. Assuming, additionally, that $f$ is a continuously differentiable function on $\Sigma_{n\times m}^{N}$, and taking its Taylor expansion, in the limit of infinitely many small players $N\to \infty$, we find that the above generator eventually converges to:

\begin{align}
\begin{gathered}\label{eq3.3}
L^{t}f(x)=\sum_{a}^{n}\sum_{\beta}^{m}\sum_{c}^{n}x_{a\beta}\cdot q_{a\beta\to c\beta}\cdot\bigl(\frac{\partial f}{\partial x_{c\beta}}-\frac{\partial f}{\partial x_{a\beta}}\bigr)\\
+\sum_{a}^{n}\sum_{\beta}^{m}\sum_{c}^{n}\sum_{\gamma}^{n}\sum_{k}^{m}x_{a\beta}\cdot \delta_{int}\cdot q_{a\beta\to c\beta}^{\gamma k}\cdot x_{\gamma k}\cdot \bigl(\frac{\partial f}{\partial x_{c\beta}}-\frac{\partial f}{\partial x_{a\beta}}\bigr)\\
+\sum_{a}^{n}\sum_{\beta}^{m}\sum_{c}^{m}x_{a\beta}\cdot \lambda\cdot u_{a\beta\to ac}\cdot \bigl(\frac{\partial f}{\partial x_{ac}}-\frac{\partial f}{\partial x_{a\beta}}\bigr),
\end{gathered}
\end{align}

\noindent or equivalently to the form:

\begin{align}
\begin{gathered}\label{eq3.4}
L^{t}f(x)=\sum_{a\neq c}^{n}\sum_{\beta}^{m}\sum_{c}^{n} (x_{a\beta}\cdot q_{a\beta\to c\beta}-x_{c\beta}\cdot q_{c\beta\to a\beta}) \cdot \frac{\partial f}{\partial x_{c\beta}}\\
+\sum_{a\neq c}^{n}\sum_{\beta}^{m}\sum_{c}^{n}\sum_{\gamma}^{n}\sum_{s}^{m} (x_{a\beta}\cdot\delta_{int}\cdot q_{a\beta\to c\beta}^{\gamma s}\cdot x_{\gamma s}-x_{c\beta}\cdot\delta_{int}\cdot q_{c\beta\to a\beta}^{\gamma s}\cdot x_{\gamma s}) \cdot \frac{\partial f}{\partial x_{c\beta}} \\
+\sum_{\beta}^{m}\sum_{c}^{n}\sum_{s\neq\beta}^{m} (x_{cs}\cdot\lambda\cdot u_{cs\to c\beta}-x_{c\beta}\cdot \lambda\cdot u_{c\beta\to cs}) \cdot \frac{\partial f}{\partial x_{c\beta}}.
\end{gathered}
\end{align}

This is a first order partial differential operator, that generates a deterministic Markov process, whose dynamics are governed by the characteristic equations of $L^{t}$:

\begin{align}
\begin{gathered}\label{eq3.5}
\dot x_{ij}=\sum_{k\neq j}^{m} (x_{ik}\cdot \lambda\cdot u_{ik\to ij}-x_{ij}\cdot \lambda\cdot u_{ij\to ik}) + \sum_{a\neq i}^{n} (x_{aj}\cdot q_{aj\to ij}-x_{ij}\cdot q_{ij\to aj})\\+\sum_{k}^{m}\sum_{a\neq i}^{n}\sum_{\gamma}^{n} (x_{aj}\cdot\delta_{int}\cdot q_{sj\to ij}^{\gamma k}\cdot x_{\gamma k}-x_{ij}\cdot \delta_{int}\cdot q_{ij\to aj}^{\gamma k}\cdot x_{\gamma k}).
\end{gathered}
\end{align}

These calculations make the following result plausible:

\begin{proposition}\label{prop3.5}
Given the Markovian interaction we introduced above consisting of the decision, the pressure-resistance and the evolution structures, if the elements of the matrix-valued function $x=(x_{ij})$ denote the occupation probabilities of states $(i,j)$, and $(u_{i,k\to j})$ is the decision matrix that may depend on time, the evolution of $x$ is given by system \eqref{eq3.5}.
\end{proposition}

\begin{remark}\label{re3.4}
For a rigorous explanation (not just the formal description that we provide here) of the Markov chain's convergence to the deterministic process given by \eqref{eq3.5}, see, e.g., \cite{Kol-2}.
\end{remark}

The above general structure is rather complicated. To deal effectively with this complexity, one can distinguish two natural simplifying frameworks: (i) the set of edges is ordered and only the transitions between neighbours are allowed, (ii) the corresponding graph is complete, so that all  transitions are allowed and have comparable rates. We shall choose the second alternative for $B$, and the first alternative for $H$ thinking of it as an hierarchy of agents. Moreover, we shall assume that the binary interaction occurs only within a common level in $H$, ignoring the binary interaction between the agents in different levels of the hierarchy structure. Therefore, for the transition rates $q_{ij\to i+1,j}$ of the pressure structure increasing in $i\in H$, we introduce the shorter notation $q^+_{ij}$, and  for the transition rates $q_{ij\to i-1,j}$ decreasing in $i\in H$, we introduce the notation $q^-_{ij}$. Accordingly, for the transition rates $q^{ik}_{ij\to i+1,j}$ of the hierarchical evolution structure increasing in $i\in H$, we introduce the shorter notation $q^{+k}_{ij}$, and  for the transition rates $q^{ik}_{i\to i-1,j}$ decreasing in $i\in H$, we shall use the notation $ q^{-k}_{ij}$.

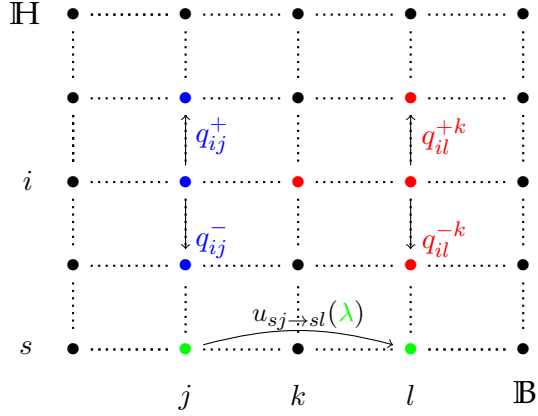
\begin{figure}[H]\centering
\begin{tikzpicture}[scale=0.74]
\node [] (E) at (8,4.5) {\textbullet};
\node[] (D) at (6,4.5) {\color{red}\textbullet};
\node [] (C) at (4,4.5) {\textbullet};
\node [] (B) at (2,4.5) {\color{blue}\textbullet};
\node [] (A) at (0,4.5) {\textbullet};
\node [] (J) at (8,3) {\textbullet};
\node [] (I) at (6,3) {\color{red}\textbullet};
\node [] (H) at (4,3) {\color{red}\textbullet};
\node [] (G) at (2,3) {\color{blue}\textbullet};
\node [] (F) at (0,3) {\textbullet};
\node [] (O) at (8,1.5) {\textbullet};
\node [] (N) at (6,1.5) {\color{red}\textbullet};
\node [] (M) at (4,1.5) {\textbullet};
\node [] (L) at (2,1.5) {\color{blue}\textbullet};
\node [] (K) at (0,1.5) {\textbullet};
\node [] (S6) at (8,6) {\textbullet};
\node [] (S5) at (6,6) {\textbullet};
\node [] (S4) at (4,6) {\textbullet};
\node [] (S3) at (2,6) {\textbullet};
\node [] (S2) at (0,6) {\textbullet};
\node [] (S1) at (8,0) {\textbullet};
\node [] (S) at (6,0) {\color{green}\textbullet};
\node [] (R) at (4,0) {\textbullet};
\node [] (Q) at (2,0) {\color{green}\textbullet};
\node [] (P) at (0,0) {\textbullet};
\path[->,every node/.style={font=\sffamily\small}]
(Q) edge[bend left=15] node [xshift=29,left,yshift=6] {$u_{sj\to sl}(\color{green}\lambda\color{black})$} (S);

\draw[dotted,thick] (A) -- (B) -- (C) -- (D) -- (E);
\draw[dotted,thick] (F) -- (G) -- (H) -- (I) -- (J);
\draw[dotted,thick](K) -- (L) -- (M) -- (N) -- (O);
\draw[<-] (D) -- (I) node[xshift=25,left,yshift=17.5] {$\color{red}q_{il}^{+k}$};
\draw[->] (I) -- (N) node[xshift=25,left,yshift=10] {$\color{red}q_{il}^{-k}$};

\draw[<-] (B) -- (G) node[xshift=20,left,yshift=17.5] {$\color{blue}q_{ij}^{+}$};
\draw[->] (G) -- (L) node[xshift=20,left,yshift=10] {$\color{blue}q_{ij}^{-}$};

\draw[dotted,thick] (S3) -- (S2) node[xshift=-8,left,yshift=0.5] {\large$\mathbb{H}$};
\draw[dotted,thick] (Q) -- (P) node[xshift=-11,left,yshift=0.5] {$s$};
\draw[dotted,thick] (A) -- (F) node[xshift=-11,left,yshift=0.5] {$i$};
\draw[dotted,thick] (A) -- (F) node[pos=0,left,yshift=6] {};
\draw[dotted,thick] (Q) -- (L) node[pos=-0.4,below,yshift=-8] {$j$};
\draw[dotted,thick] (R) -- (M) node[pos=-0.4,below,yshift=-8] {$k$};
\draw[dotted,thick] (S) -- (N) node[pos=-0.4,below,yshift=-8] {$l$};

\draw[dotted,thick] (S1) -- (S) node[pos=-0.25,below,yshift=-8] {\large$\mathbb{B}$};

\draw[dotted,thick] (A) -- (F) node[pos=0,left,yshift=-45] {};
\draw[dotted,thick] (S2) -- (A) -- (F) -- (K) -- (P);
\draw[dotted,thick] (S3) -- (B) -- (G) -- (L);
\draw[dotted,thick] (S4) -- (C) -- (H) -- (M);
\draw[dotted,thick] (S5) -- (D) -- (I) -- (N);
\draw[dotted,thick] (S6) -- (E) -- (J) -- (O) -- (S1) ;
\draw[dotted,thick] (P) -- (Q) -- (R) -- (S) -- (S1);
\draw[dotted,thick] (S2) -- (S3) -- (S4) -- (S5) -- (S6);
\end{tikzpicture}
\caption{The simplified version of our network: only the transitions between neighbours are allowed in $H$, all transitions are allowed in $B$, binary interaction occurs only within a common level in $H$.}
\end{figure}

Applying the above simplifications, the kinetic equations \eqref{eq3.5} reduce to the following system:

\begin{align}
\begin{gathered}\label{eq3.6}
\dot x_{ij}=\lambda\cdot \sum_{k\neq j} (u_{ik\to ij} x_{ik} -u_{ij\to ik} x_{ij})
+q^-_{i+1,j}\cdot x_{i+1,j}+q^+_{i-1,j}\cdot x_{i-1,j}-(q^+_{ij}+q^-_{ij})\cdot x_{ij}\\
+\delta_{int}\cdot\sum_{k\in B}(q^{+k}_{i-1,j}\cdot x_{i-1,k}\cdot x_{i-1,j} + q^{-k}_{i+1,j}\cdot x_{i+1,k}\cdot x_{i+1,j}-(q^{+k}_{ij}+q_{ij}^{-k})\cdot  x_{ik}\cdot  x_{ij}).
\end{gathered}
\end{align}

Note that equations \eqref{eq3.6} hold only for the internal states $(i,j)$, $i\in H$, $j\in B$, such that $i\neq 1, |H|$. On the contrary, for the boundary states $(i,j)$ the terms involving downgrading to $i-1$ and upgrading to $i+1$ respectively are omitted. In particular, we consider:

\begin{equation}\label{eq3.7}
q_{nj}^{+k}=q_{1j}^{-k}=0\quad,\quad q_{nj}^{+}=q_{1j}^{-}=0.
\end{equation}

Additionally, to simplify further the final explicit calculations, for all $i\in H$, $j\in B$, we shall consider the constraint:

\begin{equation}\label{eq3.8}
q_{ij}^+=q^-_{i+1,j},
\end{equation}

\noindent which can be interpreted as a \textit{detailed balance condition}; it actually asserts that the number of downgrades is compensated in average by the number of upgrades.

\begin{remark}\label{re3.5}
An alternative simple (and analogously manageable) model allows the principal either to move an agent one-step
upward in the hierarchy with rates $q^+_{ij}$ and $q^{+k}_{ij}$ respectively, or send an agent directly
down to the lowest state with rates $q^d_{ij}$ and $q^{dk}_{ij}$ respectively. In this case the system describing the evolution of occupation densities becomes, for $i\neq 1$:

\begin{align}
\begin{gathered}\label{eq3.9}
\dot x_{ij}=\lambda\cdot \sum_{k\neq j} (u_{i,k\to j}\cdot x_{ik} -u_{i,j\to k}\cdot x_{ij})
+q^+_{i-1,j}\cdot x_{i-1,j}-(q^+_{ij}+q^d_{ij})\cdot x_{ij}\\
+\delta_{int}\cdot\sum_{k\in B}(q^{+k}_{i-1,j}\cdot x_{i-1,k}\cdot x_{i-1,j}-q^{+k}_{ij}\cdot x_{ik}\cdot x_{ij})  -\delta_{int}\cdot\sum_{k\in B}q^{dk}_{ij}\cdot x_{ik}\cdot x_{ij},
\end{gathered}
\end{align}

\noindent with an obvious modification for $i=1$.
\end{remark}

To identify the agents' optimal decision vector, we need first to define certain game characteristics such as the state rewards and the transition costs. In particular, we assign the reward $w_{ij}$ per unit of time to an agent for staying in state $(i,j)$, the fee/cost $f^B_{kj}$ for an agent's elective transition from state $(h,k)$ to state $(h,j)$ (which we assume independent of $h$ for brevity), and the fine/cost $f^H_j$ for an agent's enforced transition from state $(j,b)$ to state $(j-1,b)$ (which we assume independent of $b$ for brevity). Let, additionally, $g_{ij}=g_{ij}(t)$ be the payoff corresponding to the state $(i,j)$ in the process starting at time $t$ and terminating at time $T$. Then, for an infinitesimally small time step $\tau$, and assuming that $g(t)$ is differentiable in time, an agent at state $(i,j)$ decides his/her strategy targeting to optimize the expression:

\begin{align}
\begin{gathered}\label{eq3.10}
g_{ij}(t)=\max_{u}\bigl\{\tau\cdot w_{ij}+\tau\cdot \bigl(\lambda\cdot u_{ij\to ik}\cdot (g_{ik}(t+\tau)-f_{jk}^{B})+q_{ij}^{-} \cdot(g_{i-1,j}(t+\tau)-f_{i}^{H})\\+q_{ij}^{+}\cdot g_{i+1,j}(t+\tau)+\sum_{k}^{m}x_{ik}\cdot\delta_{int}\cdot (q_{ij}^{+k} g_{i+1,j}(t+\tau)+q_{ij}^{-k} (g_{i-1,j}(t+\tau)-f_{i}^{H}))\bigr)\vspace{-1.5mm}\\+\bigl(1-\tau\cdot (\lambda\cdot u_{ij\to ik}+q_{ij}^{+}+q_{ij}^{-}+\sum_{k}^{m}x_{ik}\cdot \delta_{int}\cdot (q_{ij}^{+k}+q_{ij}^{-k}))\bigr)\cdot g_{ij}(t+\tau)\bigr\}.
\end{gathered}
\end{align}

\begin{remark}\label{re3.6}
Depending on the application we choose to investigate in each instance, the agents' optimum can be either to maximize his/her payoff/fitness, or to minimize his/her cost. Here we stick to the first case, thinking of bribed bureaucrats or defending~civilians.
\end{remark}

Taking the Taylor expansion specifically of the term $g_{ij}(t+\tau)$, and omitting terms of order $O(\tau^{2})$, the~above optimization equation turns into the form:

\begin{align}
\begin{gathered}\label{eq3.11}
w_{ij}+\frac{\partial g_{ij}(t)}{\partial t}+\max_{u}\bigl\{\lambda\cdot u_{ij\to ik}\cdot (g_{ik}(t+\tau)-f_{jk}^{B}-g_{ij}(t))\bigr\}\\+\sum_{k}^{m}x_{ik}\cdot\delta_{int}\cdot \bigl(q_{ij}^{+k}\cdot (g_{i+1,j}(t+\tau)-g_{ij}(t))+q_{ij}^{-k}\cdot (g_{i-1,j}(t+\tau)-f_{i}^{H}-g_{ij}(t))\bigr)\\+q_{ij}^{+}\cdot (g_{i+1,j}(t+\tau)-g_{ij}(t))+q_{ij}^{-}\cdot (g_{i-1,j}(t+\tau)-f_{i}^{H}-g_{ij}(t))=0.
\end{gathered}
\end{align}

In the limit of infinitesimally small time step $\tau\to 0$, equation \eqref{eq3.11} implies the evolutionary Hamilton-Jacoby-Bellman (HJB) equation, that is satisfied by the individual optimal payoffs $g_{ij}$. A rigorous derivation of the HJB equation can be found in every standard textbook on dynamic programming and optimal control, see, e.g., \cite{kamsc}. In particular for stochastic dynamic programming, see, e.g., \cite{ros}.

The above yields the following result:

\begin{proposition}\label{prop3.6}
Given the Markovian interaction we introduced above consisting of the decision, the pressure-resistance and the evolution structure, if $g_{ij}=g_{ij}(t)$ denotes the payoff to an agent at state $(ij)$ in the process starting at time $t$ and terminating at time $T$, and subject to a given evolution of the occupation density vector $x$ given by \eqref{eq3.6}, these individual optimal payoffs $g_{ij}(t)$ will satisfy the following evolutionary HJB equation:

\begin{align}
\begin{gathered}\label{eq3.12}
\dot g_{ij} + \lambda\cdot \max_u \{u_{ij\to ik}\cdot (g_{ik}-g_{ij}-f^B_{jk})\}+q^+_{ij}\cdot (g_{i+1,j}-g_{ij})+q^-_{ij}\cdot (g_{i-1,j}-g_{ij}-f_i^H)\\
+\delta_{int}\cdot \bigl(\sum_{k\in B}q^{+k}_{ij}\cdot x_{ik}\cdot (g_{i+1,j}-g_{ij})  +\sum_{k\in B}q^{-k}_{ij}\cdot x_{ik}\cdot (g_{i-1,j}-g_{ij}-f_i^H) \bigr)+ w_{ij}=0.
\end{gathered}
\end{align}
\end{proposition}

As above, note that equations \eqref{eq3.12} hold only for the internal states $(i,j)$, $i\in H$, $j\in B$, such that $i\neq 1, |H|$. For the boundary states $(i,j)$ the terms involving transitions to $i-1$ or from $i+1$ respectively are omitted. Indicatively, for $i=1$ it is:

\begin{equation*}\label{eq3.13}
\dot{g}_{1j}=w_{1j}+\lambda\max_u\{ u_{1,j\to k} (g_{ik}-g_{1j}-f_{jk}^{B})\}+q_{1j}^{+} (g_{2,j}-g_{1,j})+\delta_{int}\sum_{k\in B}q_{1j}^{+k} x_{ik} (g_{2j}-g_{1j}).
\end{equation*}

We consider here the optimization problem of estimating the discounted optimal payoff (alternatively one can look for the average payoff in a long time horizon). Hence, assuming the discounting coefficient $\delta_{dis}$ for future payoffs, the evolutionary HJB equation for the discounted optimal payoff $e^{-\delta_{dis}\cdot t}\cdot g_{ij}(t)$ of an agent at state $(i,j)$, with any finite planning horizon $T$, can be written as:

\begin{align}
\begin{gathered}\label{eq3.14}
\dot g_{ij} + \lambda\cdot \max_u \{u_{ij\to ik}\cdot (g_{ik}-g_{ij}-f^B_{jk})\}+q^+_{ij} (g_{i+1,j}-g_{ij})+q^-_{ij} (g_{i-1,j}-g_{ij}-f_i^H)\\
+\sum_{k\in B}\delta_{int}\cdot x_{ik}\cdot \bigl(q^{+k}_{ij} (g_{i+1,j}-g_{ij})+q^{-k}_{ij} (g_{i-1,j}-g_{ij}-f_i^H) \bigr)+ w_{ij}=\delta_{dis}\cdot g_{ij}(t).
\end{gathered}
\end{align}

The basic mean-field game consistency problem states that, for some interval $[0,T]$, every agent will benefit from applying the same common control, that is, from adopting the same decision vector. In other words, the MFG consistency condition states that one needs to consider the kinetic equations \eqref{eq3.6} (i.e. the forward system), where the collective control is taken into account, and the evolutionary HJB equations \eqref{eq3.14} (i.e. the backward system), where individual controls are used, as a coupled  forward-backward system of equations on a given time horizon $[0,T]$, complemented by some initial condition $x^0$ for the occupation density vector $x$, and some terminal condition $g_{T}$ for the optimal payoff $g$, such that $x,g$ and the common $u$ solve the aforesaid system. Our aim here is first to identify the solution of the stationary consistency problem, and then to investigate the general time-dependent problem, extending (if possible) our findings for the stationary problem. As mentioned, we shall work in three asymptotic regimes; fast execution of the agents' personal decisions, weak binary interactions, and small payoff discounting in time. 

\section{Stationary Problem}\label{sec3.3}

The stationary MFG consistency problem consists of the stationary HJB equation for the discounted optimal payoff $e^{-\delta_{dis}\cdot t}\cdot g_{ij}$ of an agent at state $(i,j)$, with a finite time horizon:

\begin{align}
\begin{gathered}\label{eq3.15}
w_{ij}+ \lambda\cdot \max_u u_{i,j\to k}\cdot (g_{ik}-g_{ij}-f^B_{jk})+q^+_{ij}\cdot (g_{i+1,j}-g_{ij})+q^-_{ij}\cdot (g_{i-1,j}-g_{ij}-f_i^H)\\
+\delta_{int}\cdot\sum_{k\in B}x_{ik}\cdot \bigl(q^{+k}_{ij}\cdot (g_{i+1,j}-g_{ij})+q^{-k}_{ij}\cdot (g_{i-1,j}-g_{ij}-f_i^H)\bigr)
=\delta_{dis}\cdot g_{ij},
\end{gathered}
\end{align}

\noindent where the evolution given by \eqref{eq3.6} is replaced with the corresponding fixed point condition:

\begin{align}
\begin{gathered}\label{eq3.16}
\lambda\cdot \sum_{k\neq j} (u_{i,k\to j}\cdot x_{ik} -u_{i,j\to k}\cdot x_{ij})
+q^-_{i+1,j}\cdot x_{i+1,j}+q^+_{i-1,j}\cdot x_{i-1,j}-(q^+_{i,j}+q^-_{i,j})\cdot x_{ij}\\
+\delta_{int}\cdot\sum_{k\in B}q^{+k}_{i-1,j}\cdot x_{i-1,k}\cdot x_{i-1,j}+ q^{-k}_{i+1,j}\cdot x_{i+1,k}\cdot x_{i+1,j}-(q^{+k}_{ij}+q_{ij}^{-k})\cdot x_{ik}\cdot x_{ij}=0.
\end{gathered}
\end{align}

By analogy with the time-dependent problem, for the stationary MFG consistency problem one needs to consider \eqref{eq3.15},\eqref{eq3.16} as a coupled stationary system. In the asymptotic limit of fast execution of individual decisions, $\lambda\to\infty$, the terms in \eqref{eq3.15},\eqref{eq3.16} containing the transition rates $\lambda$ should obviously vanish (otherwise they would `explode' to infinity). For a practical interpretation of this observation, one can think that if the execution of personal decisions is significantly fast, then in a stationary state no agent should be interested in switching strategy. In this case \eqref{eq3.15},\eqref{eq3.16} turn respectively into the form:

\begin{align}
\begin{gathered}\label{eq3.17}
w_{ij}+q^+_{ij}\cdot (g_{i+1,j}-g_{ij})+q^-_{ij}\cdot (g_{i-1,j}-g_{ij}-f_i^H)\\
+\delta_{int}\cdot\sum_{k\in B}x_{ik}\cdot \bigl(q^{+k}_{ij}\cdot (g_{i+1,j}-g_{ij})+q^{-k}_{ij}\cdot (g_{i-1,j}-g_{ij}-f_i^H)\bigr)
=\delta_{dis}\cdot g_{ij},
\end{gathered}
\end{align}

\noindent and,

\begin{align}
\begin{gathered}\label{eq3.18}
q^-_{i+1,j}\cdot x_{i+1,j}+q^+_{i-1,j}\cdot x_{i-1,j}-(q^+_{i,j}+q^-_{i,j})\cdot x_{ij}\\+\delta_{int}\cdot\sum_{k\in B}q^{+k}_{i-1,j}\cdot x_{i-1,k}\cdot x_{i-1,j}+ q^{-k}_{i+1,j}\cdot x_{i+1,k}\cdot x_{i+1,j}-(q^{+k}_{ij}+q_{ij}^{-k})\cdot  x_{ik}\cdot x_{ij}=0,
\end{gathered}
\end{align}

\noindent supplemented by the  consistency condition:

\begin{equation}\label{eq3.19}
g_{ik}-g_{ij}-f_{jk}^{B}\leq 0,
\end{equation}

\noindent for all $i\in H$, $j,k\in B$, such that $x_{ij}\neq 0$. In fact, the consistency condition \eqref{eq3.19} ensures that all terms in \eqref{eq3.15} and \eqref{eq3.16} including elements of the decision matrix  indeed vanish in \eqref{eq3.17} and \eqref{eq3.18} respectively, for all the occupied states. 

Introducing further the auxiliary notation $\tilde{w}_{ij}=w_{ij}-q_{ij}^{-}\cdot f_{i}^{H}$, \eqref{eq3.17} and \eqref{eq3.18} are written respectively in the form:

\begin{equation}\label{eq3.20}
(-A^{T}_{j}+\delta_{dis}-\delta_{int}\cdot E_{j}^{T}(x))\cdot g_{ij}=\tilde{w}_{ij}-\delta_{int}\cdot f_{i}^{H}\cdot \sum_{k\in B}q_{ij}^{-k}\cdot x_{ik},
\end{equation}

\noindent and,

\begin{equation}\label{eq3.21}
(A_{j}+\delta_{int}\cdot E_{j}(x))\cdot x_{ij}=0,
\end{equation}

\noindent where the matrices $A_{j}$, with the transpose matrix $A_{j}^{T}$, and $E_{j}(x)$, with the transpose matrix $E^{T}_{j}(x)$, are given respectively by:

\begin{equation}\label{eq3.22}
A_{j}=\begin{pmatrix}
-q_{1j}^{+} & q_{2j}^{-} & 0 & \dots \\
q_{1j}^{+} & -q_{2j}^{+}-q_{2j}^{-} & q_{3j}^{-} & \dots \\
\dots & \dots & \dots & \dots\\
\dots & q_{n-2,j}^{+} & -q_{n-1,j}^{+}-q_{n-1,j}^{-} & q_{nj}^{-}\\
\dots & 0 & q_{n-1,j}^{+} & -q_{nj}^{-}
\end{pmatrix},
\end{equation}

\noindent and,

\begin{equation}\label{eq3.23}
E_{j}=
\begingroup 
\setlength\arraycolsep{1.5pt}
\begin{pmatrix}
-\sum\limits_{k}q_{1j}^{+k} x_{1k} \hspace{-14mm}& \sum\limits_{k}q_{2j}^{-k} x_{2k} & 0 & \dots \\
\sum\limits_{k}q_{1j}^{+k} x_{1k} & -\sum\limits_{k}(q_{2j}^{+k}+q_{2j}^{-k}) x_{2k} & \sum\limits_{k}q_{3j}^{-k} x_{3k} & \dots \\
\dots & \dots & \dots & \dots\\
\dots & \sum\limits_{k}q_{n-2,j}^{+k} x_{n-2,k} & -\sum\limits_{k}(q_{n-1,j}^{+k}+q_{n-1,j}^{-k}) x_{n-1,k} & \sum\limits_{k}q_{nj}^{-k} x_{nk}\\
\dots & 0 & \sum\limits_{k}q_{n-1,j}^{+k} x_{n-1,k} & -\sum\limits_{k}q_{nj}^{-k} x_{nk}
\end{pmatrix}.
\endgroup
\end{equation}

We shall look further for the asymptotic regime with small rates $\delta_{int}\cdot q_{ij}^{\pm k}$. Therefore, starting with \eqref{eq3.21} we are looking for stationary solutions of the form:

\begin{equation}\label{eq3.24}
x_{ij}=x_{ij}^{0}+\delta_{int}\cdot x_{ij}^{1}+O(\delta_{int}^{2}).
\end{equation}

Substituting \eqref{eq3.24} into \eqref{eq3.21}, and equating terms of the same order in $\delta_{int}^{0}$, $\delta_{int}^{1}$, we obtain respectively the following equations:

\begin{equation}\label{eq3.25}
A_{j}\cdot x_{ij}^{0}=0,
\end{equation}
\begin{equation}\label{eq3.26}
A_{j}\cdot x_{ij}^{1}+E_{j}^{0}\cdot x_{ij}^{0}=0,
\end{equation}

\noindent where the notation $E^{0}_{j}$ corresponds to the matrix $E_{j}$ containing only elements of order $O(\delta_{int}^{0})$ (we use respectively the notation $E^{0T}_{j}$ for the transpose matrix).

\begin{assumption}\label{as3.2}
Let the detailed balance condition \eqref{eq3.8} hold with all $q_{ij}^{+}$ (or $q_{ij}^{-}$) being strictly positive. We shall use the shorter notation, for $i\in H : i\neq n$, $j\in B$:

\begin{equation}\label{eq3.27}
q_{ij}=q_{ij}^{+}=q_{i+1,j}^{-}.
\end{equation}
\end{assumption}

In the linear approximation of vanishing $\delta_{int}$, we end up with an uncoupled system. Since different elements of $B$ are also uncoupled, then \eqref{eq3.25},\eqref{eq3.26} can be solved separately for any $j\in B$. Looking at the zero order of small evolution transition rates, by \eqref{eq3.25}, we have the following result:

\begin{proposition}\label{prop3.7}
Let Assumption \ref{as3.2} hold. Then, the rank of $A_j$ is exactly $n-1$, while the kernel of $A_{j}$ is generated by the following vector:

\begin{align}
\begin{gathered}\label{eq3.28}
x^{0}_{2j}=\frac{q_{1j}^+}{q_{2j}^-}\cdot x^{0}_{1j}, \quad x^{0}_{3j}=\frac{q_{2j}^+}{q_{3j}^-}\cdot\frac{q_{1j}^+}{q_{2j}^-} \cdot x^{0}_{1j}, \quad \dots ,
\quad x^{0}_{nj}= \prod_{l=1}^{n-1} \frac{q_{lj}^+}{q^-_{l+1,j}}\cdot x^{0}_{1j}\\
x^{0}_{1j}=\left( 1+\frac{q_{1j}^+}{q_{2j}^-} +\frac{q_{2j}^+}{q_{3j}^-}\cdot \frac{q_{1j}^+}{q_{2j}^-} + \cdots
+ \prod_{l=1}^{n-1} \frac{q_{lj}^+}{q^-_{l+1,j}}\right)^{-1}\cdot x^{0}_j,
\end{gathered}
\end{align}

\noindent where we have introduced the notation $x^{0}_j=\sum_i x^{0}_{ij}$. Specifically, under the detailed balance condition $A_{j}$ is symmetric, and its kernel generated by \eqref{eq3.28} is proportional to the uniform distribution, $x^0_{ij}=x^0_j/n$ for all $i\in H$, $j\in B$, that is, $Ker(A_{j})$ is generated by $(1,\dots,1)$.
\end{proposition}

\begin{proof}
Notice that system \eqref{eq3.25} is degenerate, as expected, since we are looking for non-negative solutions satisfying $\sum_{j}(x^{0}_{1j}+\dots+x^{0}_{nj})=1$. Thus, one of the $n$ equations of \eqref{eq3.25} can be discarded, say for example the last one. Rewriting the system of the remaining $(n-1)$ equations by using the first equation, and then adding sequentially to each of the next $(n-2)$ equations their previous one, one eventually obtains the following system:

\begin{align}\label{eq3.29}
\begin{gathered}
 q^+_{1j}\cdot x^{0}_{1j}-q^-_{2j}\cdot x^{0}_{2j} =0 \\
 q^+_{2j}\cdot x^{0}_{2j}-q^-_{3j}\cdot x^{0}_{3j} =0 \\
  ... \\
 q^+_{n-1,j}\cdot x^{0}_{n-1,j}-q^-_{nj}\cdot x^{0}_{nj} =0 .
\end{gathered}
\end{align}

This has an obvious solution, that is unique up to a multiplier, and is given by $\eqref{eq3.28}$. Alternatively, starting the exclusion from the last equation of \eqref{eq3.29}, the solution to \eqref{eq3.25}~is:

\begin{align}
\begin{gathered}\label{eq3.30}
x^0_{n-1,j}=\frac{q_{nj}^-}{q_{n-1,j}^+}\cdot x^{0}_{nj}, \enskip x^0_{n-2,j}=\frac{q_{n-1,j}^-}{q_{n-2,j}^+}\cdot \frac{q_{nj}^-}{q_{n-1,j}^+}\cdot x^0_{nj}, \enskip \dots ,
\enskip x^0_{1j}= \prod_{l=1}^{n-1} \frac{q_{l+1,j}^-}{q^+_{l,j}}\cdot x^0_{nj}\\
x^0_{nj}=\left( 1+ \frac{q_{nj}^-}{q_{n-1,j}^+} +\frac{q_{n-1,j}^-}{q_{n-2,j}^+}\cdot \frac{q_{nj}^-}{q_{n-1,j}^+} + \dots
+ \prod_{l=1}^{n-1} \frac{q_{l+1,j}^-}{q^+_{l,j}} \right)^{-1} \cdot x^0_j.
\end{gathered}
\end{align}

Given now the detailed balance condition \eqref{eq3.8}, and the non-degeneracy~established by Assumption \ref{as3.2}, one observes from \eqref{eq3.28}, or \eqref{eq3.30}, that for every strategy $j\in B$ we have: 

$$x^{0}_{1j}=x^{0}_{2j}=\dots=x^{0}_{nj}=x^{0}_{j}/n.$$
\end{proof}

We have shown that in the main order of small evolution rates $\delta_{int}\cdot q_{ij}^{\pm \kappa}$, $x_{ij}^{0*}=x_{j}^{0*}/n$ is a fixed point of the evolution \eqref{eq3.6}, along with the common control $u^{com}=(u_{ij\to i\kappa}=0)$, $\forall i\in H$, $\forall j,\kappa\in B$, that is consistent with condition \eqref{eq3.19}, and expresses the instantaneous execution of personal decisions. This will also be a stable solution of the stationary system, if $x_{ij}^{0*}=x_{j}^{0*}/n$ is a stable fixed point of \eqref{eq3.6}, for $u^{com}=(u_{ij\to i\kappa}=0)$, $\forall i\in H, j,\kappa\in B$.

\begin{assumption}\label{as3.3}
For technical (computational) purposes only, let the hierarchy and the strategy set be of the same size, i.e. $|H|=|B| \Rightarrow n=m$.
\end{assumption}

To conduct the stability analysis, in the asymptotic regimes of large $\lambda$ and small $\delta_{int}$, let us introduce the auxiliary variables:

\begin{equation}\label{eq3.31}
y_{\kappa}=x_{ij}^{0}-x_{ij}^{0*},
\end{equation}

\noindent where $\kappa=i+(j-1)\cdot n$, such that $\kappa\in K=\{1,\dots,n^{2}-1\}$. 

Using the above variables, we transform system \eqref{eq3.6} into the non-degenerate linear homogeneous system:

\begin{equation}\label{eq3.32}
\dot{y}=\Lambda\cdot y,
\end{equation}

\noindent where $\Lambda$ is the block matrix:

\begin{equation}\label{eq3.33}
\Lambda=\begin{pmatrix}
A_{1} & 0 & \dots  & \dots  & \dots  \\
0 & A_{2} & 0 & \dots & \dots  \\
\dots & 0 & A_{j} & 0 & \dots \\
\dots & \dots & 0 & A_{n-1} & 0\\
\Delta & \dots & \dots & \Delta & D
\end{pmatrix}.
\end{equation}

Each matrix $\Delta$ has the same non zero entries $-q_{nn}^{-}$ on its bottom row, while the rest of its elements are equal to zero. Note, as well, that the $A_{j}$ matrices are of dimension $n\times n$, and each zero matrix to the right of an $A_{j}$ matrix is of dimension $n\times n\cdot (n-j)-1$. The $(n-1)\times (n-1)$ matrix $D$ is given by:

\begin{equation}\label{eq3.34}
D=\begin{pmatrix}
-q_{1n}^{+} & q_{2n}^{-} & 0  & \dots  \\
q_{1n}^{+} & -q_{2n}^{+}-q_{2n}^{-} & q_{3n}^{-} & \dots  \\
\dots & \dots & \dots & \dots \\ 
\dots & q_{n-3,n}^{+} & -q_{n-2,n}^{+}-q_{n-2,n}^{-} & q_{n-1,n}^{-}\\
\dots & 0 & q_{n-2,n}^{+}-q_{nn}^{-} & -q_{n-1,n}^{+}-q_{n-1,n}^{-}-q_{nn}^{-}
\end{pmatrix}.
\end{equation}

Applying sequentially (starting with $C_{1}\equiv A_{1}$, setting in the next step $C_{1}\equiv A_{2}$ etc.) the following block matrix formula:

\[
\det
\left( {\begin{array}{cc}
C_{1} & 0 \\
C_{2} & C_{3}
\end{array} } \right)=\det C_{1}\cdot \det C_{3},
\]

\noindent where $C_{1}$, $C_{2}$, and $C_{3}$ are $n \times n$, $m \times n$, and $m \times m$ matrices respectively, the determinant of $\Lambda$ is given by:

\begin{equation}\label{3.35}
\det\Lambda= \det (A_{1})\cdot \det(A_{2})\cdots\det(A_{n-1})\cdot \det D.
\end{equation}

We further apply sequentially $n-1$ times the elementary row operation of row addition on every $n \times n$ matrix $A_{j}$, starting with row $n$ and adding in each step row $i$ to row $i-1$. 

Eventually, we transform $A_{j}$ into a lower triangular matrix of the form:

\begin{equation}\label{eq3.36}
\begin{pmatrix}
0 & 0 & 0 & \dots \\
q_{1j}^{+} & -q_{2j}^{-} & 0 & \dots \\
\dots & \dots & \dots & \dots\\
\dots & q_{n-2,j}^{+} & -q_{n-1,j}^{-} & 0\\
\dots & 0 & q_{n-1,j}^{+} & -q_{nj}^{-}
\end{pmatrix},
\end{equation}

\noindent with a single zero eigenvalue, and $n-1$ negative eigenvalues $-q_{ij}^{-}$, for $i=2,\dots,n$. Note that since $A_{j}$ are symmetric matrices (due to the detailed balance condition), the algebraic multiplicity of each of their eigenvalues is equal to the geometric multiplicity. 

Regarding the $(n-1)\times (n-1)$ matrix $D$, and bearing in mind the detailed balance condition, we apply once the elementary row operation of adding row $n-1$ to row $n-2$, and then, we apply sequentially $n-2$ times the elementary column operation of adding column $i$ to column $i+1$, starting with column $1$, to eventually transform $D$ into the following lower triangular form:

\begin{equation}\label{eq3.37}
\begin{pmatrix}
-q_{1n}^{+} & 0 & 0  & \dots  \\
q_{1n}^{+} & -q_{2n}^{+} & 0 & \dots  \\
\dots & \dots & \dots & \dots \\ 
\dots & q_{n-3,n}^{+} & -q_{n,n}^{-} & 0\\
\dots & 0 & (q_{n-2,n}^{+}-q_{n,n}^{-}) & -q_{n-1,n}^{-}
\end{pmatrix},
\end{equation}

\noindent with $n-1$ negative eigenvalues $-q_{in}^{+}$, for $i=1,\dots,n-1$. In total, we find that matrix $\Lambda$ has one zero eigenvalue of algebraic multiplicity $n-1$, and $n\cdot (n-1)$ negative eigenvalues. Now it is trivial to transform $\Lambda$ into a block diagonal matrix, subtracting sequentially from each column $i$, $\forall i=\{1,\dots,n\cdot n-n\}$, each column $j$, $\forall j=\{n\cdot n-n+2,\dots,n\cdot n -1\}$. For a block diagonal matrix, both the algebraic and the geometric multiplicity of an eigenvalue is given by adding the multiplicities from each block. Then, for the block matrix $\Lambda$ the algebraic multiplicity of the zero eigenvalue is equal to its geometric multiplicity. 

We, thus, have the following result:

\begin{lemma}\label{lem3.3}
Let the Assumptions \ref{as3.2}, \ref{as3.3} hold. Consider the linear system $\dot{y}=\Lambda\cdot y$ as defined above. The solution to this system, that is the vector $x_{ij}^{0*}=x_{j}^{0*}/n$ given by Proposition \ref{prop3.7}, is stable (but not asymptotically stable) since $\Lambda$ has $n\cdot (n-1)$ negative eigenvalues, and a single zero eigenvalue whose algebraic multiplicity equals to its geometric multiplicity.
\end{lemma}

The third asymptotic regime we shall look at is that of small discounting $\delta_{dis}$. Obviously, no payoff discounting terms appear in the stationary kinetic equations \eqref{eq3.18}. Moving to the stationary HJB equation \eqref{eq3.17}, or \eqref{eq3.20}, initially we are looking for solutions of the form:

\begin{equation}\label{eq3.38}
g_{ij}=g_{ij}^{0}+\delta_{dis}\cdot g_{ij}^{1}.
\end{equation}

Substituting \eqref{eq3.38} into \eqref{eq3.20}, and equating terms of zero order in $\delta_{int}$ and $\delta_{dis}$, we get the equation:

\begin{equation}\label{eq3.39}
-A^{T}_{j}\cdot g^{0}_{ij}=\tilde{w}_{ij},
\end{equation}

In general, equation \eqref{eq3.39} has no (non-degenerate) solution, since (by Proposition \ref{prop3.7}) the kernel of the symmetric matrix $A_{j}^{T}=A_{j}$ is one dimensional, implying that the image of the transpose matrix $A_{j}^{T}$ is $(n-1)$ dimensional (by the rank-nullity theorem). More precisely, equation \eqref{eq3.39} has no solution if:

\begin{equation}\label{eq3.40}
(\tilde{w}_{ij},x^{0}_{ij})=\frac{x^{0}_{j}}{n}\cdot \sum_{i}\tilde{w}_{ij}\neq 0.
\end{equation}

Thus, to remain in the non-degenerate regime, we need to introduce additionally the following assumption;

\begin{assumption}\label{as3.4}
For every strategy $j\in B$ the following is true; $\sum\limits_{i}\tilde{w}_{ij}\neq 0$.
\end{assumption}

As a result, we are looking next for solutions of \eqref{eq3.20} in the form of the expansion:

\begin{equation}\label{eq3.41}
g_{ij}=g_{ij}^{0}/\delta_{dis}+g_{ij}^{1}+g_{ij}^{2}\cdot \delta_{dis}.
\end{equation}

Recall that we are looking at the asymptotic regime with small $\delta_{int}$ (weak binary interaction), and small $\delta_{dis}$ (small discounting). One needs to distinguish clear assumptions on the relation between the small parameters $\delta_{int}$ and $\delta_{dis}$, for a full perturbation analysis. In principle, three basic regimes can be naturally identified:

\begin{itemize}[leftmargin=0.5cm]
\item[] ID$_{1}$: Interaction is relatively very small, i.e. $\delta_{dis}=\delta$ and $\delta_{int}=\delta^{2}$.
\item[] ID$_{2}$: Interaction and Discounting are small effects of comparable order, i.e. $\delta_{dis}=\delta_{int}=\delta$.
\item[] ID$_{3}$: Discounting is relatively very small, i.e. $\delta_{int}=\delta$ and $\delta_{dis}=\delta^{2}$.
\end{itemize}

We initially concentrate on the ID$_{1}$ regime. Substituting \eqref{eq3.41} into \eqref{eq3.20}, and equating terms of order $\delta^{-1}$, $\delta^{0}$, $\delta^{1}$, we find respectively the following equations:

\begin{align}
\begin{gathered}\label{eq3.42}
A_{j}^{T}\cdot g_{ij}^{0}=0
\\
-A_{j}^{T}\cdot g_{ij}^{1}+g_{ij}^{0}=\tilde{w}_{ij}
\\
-A_{j}^{T}\cdot g_{ij}^{2}+g_{ij}^{1}-E_{j}^{0T}\cdot g_{ij}^{0}=0.
\end{gathered}
\end{align}

The first equation in \eqref{eq3.42} tells us that $g_{ij}^{0}$ belongs to the kernel of $A_{j}$ (since $A_{j}=A_{j}^{T}$), that is, for arbitrary constants $a_{j}\in \mathbb{R}$, we get:

\begin{equation}\label{eq3.43}
g_{ij}^{0}=a_{j}\cdot x_{ij}^{0}.
\end{equation}

The second equation in \eqref{eq3.42} tells us that $\tilde{w}_{ij}-g_{ij}^{0}$ belongs to the image of $A_{j}$, which coincides with the orthogonal compliment to $Ker(A_{j})$, given the identity:

\begin{equation*}
Im(A_{j})=Ker^{\perp}(A_{j}^{T}).
\end{equation*}

\noindent Besides, from Proposition \ref{prop3.7} we find that the orthogonal compliment to $Ker(A_{j})$ is:

\begin{equation}\label{eq3.44}
Ker^{\perp}(A_{j})=\{x:\sum_{i}x_{ij}=0\}.
\end{equation}

\noindent In this case, the fact that $\tilde{w}_{ij}-g_{ij}^{0}\in Im(A_{j})$ further implies that:

\begin{equation}\label{eq3.45}
\sum_{i}\tilde{w}_{ij}=\sum_{i}g_{ij}^{0}\Rightarrow\dots\Rightarrow g_{ij}^{0}=\sum_{i}\tilde{w}_{ij}/n.
\end{equation}

Looking at the third equation in \eqref{eq3.42}, and noting that $E^{0T} g_{ij}^{0}=0$ for a uniform $g_{\cdot j}^{0}$, we conclude that $g_{ij}^{1}\in Im(A_{j})$ as well, that is, $g_{ij}^{1}\in Ker^{\perp}(A_{j})$. Thus, to identify $g_{ij}^{1}$ we need to invert $A_{j}$ on the reduced $(n-1)$ dimension of $Ker^{\perp}(A_{j})$.

\begin{lemma}\label{le3.4}
Let Assumption \ref{as3.2} hold, and let $y\in Ker^{\perp}(A_{j})$. Then all solutions $z$ to the matrix equation $A_{j}\cdot z=y$ are given by the formula:

\begin{equation}\label{eq3.46}
z_{ij}=z_{1j}-\sum_{a=1}^{i-1}\Bigl(\sum_{\beta=1}^{a}\frac{y_{\beta j}}{q_{aj}}\Bigr),
\end{equation}

\noindent $\forall i\neq 1$, with arbitrary $z_{1j}$. There exists a unique solution $z_{\cdot j}\in Ker^{\perp}(A_{j})$ specified by:

\begin{equation}\label{eq3.47}
z_{1j}=\sum_{a=1}^{n-1}\Bigl(\frac{n-a}{n}\cdot\sum_{\beta=1}^{a}\frac{y_{\beta j}}{q_{aj}}\Bigr).
\end{equation}
\end{lemma}

Notice that formulae \eqref{eq3.46} and \eqref{eq3.47} yield $z_{ij}=g_{ij}^{1}$ when $y_{ij}=\tilde{w}_{ij}-g_{ij}^{0}$. In particular, for $g_{ij}^{1}$ we find the explicit expression:

\begin{equation}\label{eq3.48}
g_{ij}^{1}=\sum_{a=1}^{n-1}\Bigl(\bigl(\mathbf{1}(i>a)\cdot\frac{n-a-1}{n}+\mathbf{1}(i\leq a)\cdot\frac{n-a}{n}\bigr)\cdot\bigl(\frac{a}{q_{aj}}\cdot\sum_{\kappa\in H}\frac{\tilde{w}_{\kappa j}}{n}-\sum_{\beta=1}^{a}\frac{\tilde{w}_{\beta j}}{q_{aj}}\bigr)\Bigr),
\end{equation}

\noindent where $\mathbf{1}(\cdot)$ is the indicator function.

Regarding the consistency condition \eqref{eq3.19}, in the main order in small $\delta$ it can be written in the equivalent form:

\begin{equation}\label{eq3.49}
\sum_{i}\tilde{w}_{ik}<\sum_{i}\tilde{w}_{ij},
\end{equation}

\noindent for all $i\in H$, $k,j \in B$. Given that $\tilde{w}_{\cdot,\cdot}$ does not depend on $\delta$, this leads to the  interesting result that in the equilibrium of the asymptotic regime of small $\delta$, only those strategic levels $j\in B$ are occupied (that is, $x_{j}^{0}\neq 0$), where the sum $\sum_{i}\tilde{w}_{ij}$ obtains its maximum. 

For simplicity, let us further consider the following assumption;

\begin{assumption}\label{as3.5}
There exists a unique behavioural level $b\in B$, such that:

\begin{equation}\label{eq3.50}
\sum_{i}\tilde{w}_{ib}>\sum_{i}\tilde{w}_{ij}.
\end{equation}

\noindent for all $j\in B$, such that $j\neq b$.
\end{assumption}

Note that Assumption \ref{as3.5} implies that in any equilibrium $x^{*}$, with $\delta$ sufficiently small, all terms with $j\neq b$ become irrelevant for the analysis. 

We, thus, have the following result:

\begin{proposition}\label{prop3.8}
Let Assumptions \ref{as3.2}, \ref{as3.3}, \ref{as3.4} and \ref{as3.5} hold. Consider the ID$_{1}$ regime. Then, the solution to the stationary problem described by \eqref{eq3.17}, \eqref{eq3.18} and \eqref{eq3.19}, in the main order in small $\delta$, is given by:

\begin{equation}\label{eq3.51}
x_{ib}^{*}=x_{ib}^{0*}=1/n,\quad x_{i\kappa}^{0*}=0\enskip \forall \kappa\neq b\in B, i\in H,\quad g_{ib}=\delta^{-1}\cdot g_{ib}^{0}=\delta^{-1}\cdot\sum_{i}\tilde{w}_{ib}/n,
\end{equation}

\noindent where $x_{ij}^{0*}$ is a stable fixed point of \eqref{eq3.6}.
\end{proposition}

\begin{remark}
If we continue in the next order of our perturbation analysis (subsequently in the second next order, and so forth) we can obtain explicit approximate solutions with arbitrary precision.
\end{remark}

Next we consider the ID$_{2}$ regime. In this case, we look at the solutions to \eqref{eq3.18} in the next order with respect to small $\delta$. In view of \eqref{eq3.51}, we write \eqref{eq3.26} in the form:

\begin{equation}\label{eq3.52}
A_{b}\cdot x_{ib}^{1}+(q_{i-1,b}^{+b}-q_{ib}^{+b}+q_{i+1,b}^{-b}-q_{ib}^{-b})/n^{2}=0,
\end{equation}

\noindent where $\sum_{i}x_{ib}^{1}=0$, and the usual convention for the boundary terms, $i=1,n$, apply. 

Note that the right-hand side of Equation \eqref{eq3.52} belongs to $Ker^{\perp}(A_{j})$, implying that:

\begin{equation}\label{eq3.53}
\sum_{i}-(q_{i-1,b}^{+b}-q_{ib}^{+b}+q_{i+1,b}^{-b}-q_{ib}^{-b})/n^{2}=0,
\end{equation}

Moreover, given that $x_{ib}^{1}\in Ker^{\perp}(A_{j})$, we can identify $x_{ib}^{1}$ applying Lemma \ref{le3.4}. Formulae \eqref{eq3.46}, \eqref{eq3.47} yield $z_{ib}=x_{ib}^{1}$ when $y_{ib}=-(q_{i-1,b}^{+b}-q_{ib}^{+b}+q_{i+1,b}^{-b}-q_{ib}^{-b})/n^{2}$.

Regarding the solution to \eqref{eq3.20} in ID$_{2}$, substituting \eqref{eq3.41} into \eqref{eq3.20}, and equating terms of order $\delta^{-1}$, $\delta^{0}$, $\delta^{1}$, we get respectively the following equations:

\begin{align}
\begin{gathered}\label{eq3.54}
A_{j}^{T}\cdot g_{ij}^{0}=0\\
-A^{T}_{j}\cdot g_{ij}^{1}+g_{ij}^{0}-E_{j}^{0T}\cdot g_{ij}^{0}=\tilde{w}_{ij}\\
-A_{j}^{T}\cdot g_{ij}^{2}+g_{ij}^{1}-E_{j}^{0T}\cdot g_{ij}^{1}-E_{j}^{1T}\cdot g_{ij}^{0}=-f_{i}^{H}\cdot\sum_{k}q_{ij}^{-k}\cdot x_{ik}^{0},
\end{gathered}
\end{align}

\noindent where the notation $E^{1}_{j}$ corresponds to the matrix $E_{j}$ containing only elements of order $O(\delta_{int})$ (we use respectively the notation $E^{1T}_{j}$ for the transpose matrix).

The first two equations in \eqref{eq3.54} are identical with the corresponding equations in \eqref{eq3.42} (recall that $E_{j}^{0T} g_{ij}^{0}=0$ for a uniform $g_{ij}^{0}$), and provide the same results expressed through \eqref{eq3.43}, \eqref{eq3.45}. Looking at the third equation in \eqref{eq3.54}, and noting that $E_{j}^{1T} g_{ij}^{0}=0$, we observe that $(g_{ij}^{1}-E_{j}^{0T} g_{ij}^{1}+f_{i}^{H}\cdot\sum_{k}q_{ij}^{-k}\cdot x_{ik}^{0})\in Ker^{\perp}(A_{j})$, implying that $g_{ij}^{1}$ can be uniquely identified through formula \eqref{eq3.46} of Lemma \ref{le3.4}, with $z_{ij}=g_{ij}^{1}$ and $y_{ij}=\tilde{w}_{ij}-g_{ij}^{0}$, under the condition:

\begin{equation}\label{3.55}
\sum_{i}(g_{ij}^{1}-E_{j}^{0T}\cdot g_{ij}^{1}+f_{i}^{H}\sum_{k}q_{ij}^{-k}\cdot x_{ik}^{0})=0.
\end{equation}

Last we consider the ID$_{3}$ regime. Substituting \eqref{eq3.41} into \eqref{eq3.20}, but equating now terms of order $\delta^{-2}$, $\delta^{-1}$, $\delta^{0}$, we get the equations (in analogy to \eqref{eq3.42},~\eqref{eq3.54}):

\begin{align}
\begin{gathered}\label{eq3.56}
A_{j}^{T}\cdot g_{ij}^{0}=0\\
E_{j}^{0T}\cdot g_{ij}^{0}=0\\
-A_{j}^{T}\cdot g_{ij}^{1}+g_{ij}^{0}-E_{j}^{1T}\cdot g_{ij}^{0}=\tilde{w}_{ij}.
\end{gathered}
\end{align}

Again, the first and the third equations in \eqref{eq3.56} lead to the same results with the first and the second equations in \eqref{eq3.42}, namely to \eqref{eq3.43} and \eqref{eq3.45} respectively, while the second equation in \eqref{eq3.56} always holds for a uniform $g^{0}_{ij}$.

We, thus, have the following result:

\begin{proposition}\label{le3.5}
The solution to the stationary consistency problem in the main order in small $\delta$ in ID$_{2}$ and ID$_{3}$, is the same with the one identified in Proposition \ref{prop3.8} for ID$_{1}$.
\end{proposition}

\section{Time-dependent Problem}\label{sec3.4}

The solution to a non-linear Markov game of mean-field type like the one we consider here (on a finite time horizon), defines an $epsilon$-Nash equilibrium of the corresponding game with a finite number of players, see, e.g., \cite{bas}. Having identified the solution to the stationary MFG consistency problem, we need next to look at the time-dependent consistency problem in order to validate our results for initial/terminal conditions other than those given by the solution of the stationary problem. We further need to investigate the stability of the fixed point $x_{ij}^{0*}$ (see Lemma \ref{lem3.3}) without necessarily assuming that from the very beginning all players apply the same stationary control $u^{com}=(u_{ij\to i\kappa}=0)$.

For the full time-dependent problem, the HJB equation for the discounted optimal payoff $e^{-\delta_{dis}\cdot t}\cdot g_{ij}(t)$ of an individual at state $(i,j)$ with any planning horizon $T$ is given by \eqref{eq3.14}, where the occupation density vector $x=(x_{ij})$ is also time varying. For definiteness, we shall focus on the ID$_{1}$ regime (the same method applies for ID$_{2}$, ID$_{3}$ regimes). Our aim is to show that by fixing the control $u_{i\alpha \to i\beta}=0$ in \eqref{eq3.14}, $\forall i\in H, \alpha,\beta\in B$, the solution to the resulting system:

\begin{align}
\begin{gathered}\label{eq3.58}
\dot g_{i\alpha} + w_{i\alpha}+q^+_{i\alpha}\cdot (g_{i+1,\alpha}-g_{i\alpha})+q^-_{i\alpha}\cdot (g_{i-1,\alpha}-g_{i\alpha}-f_i^H)\\
+\sum_{k\in B}\delta_{int}\cdot x_{ik}\cdot (q^{+k}_{i\alpha}\cdot (g_{i+1,\alpha}-g_{i\alpha})+q^{-k}_{i\alpha}\cdot (g_{i-1,\alpha}-g_{i\alpha}-f_i^H) )=\delta_{dis}\cdot g_{i\alpha}(t),
\end{gathered}
\end{align}

\noindent will be consistent, that is, the control $u_{i\alpha\to i\beta}=0$ will indeed give a maximum in \eqref{eq3.14} in all times.
Fixing the control $u_{i\alpha\to i\beta}=0$, $\forall i\in H,\alpha,\beta\in B$, is actually equivalent to assuming that:

\begin{equation}\label{eq3.59}
g_{i\beta}(T)-f_{\alpha\beta}^{B}\leq g_{i\alpha}(T).
\end{equation}

Our aim here is to show that starting with a terminal condition belonging to the cone defined by \eqref{eq3.59}, we shall stay inside the cone for all $t\leq T$. Therefore, it is sufficient to show that on the boundary of this cone the inverted tangent vector of \eqref{eq3.58} is never directed outside the cone. The necessary condition that needs to be satisfied for this to be true for any boundary point $g_{j\beta}-f_{\alpha\beta}^{B}= g_{j\alpha}$ is the following:

\begin{equation}\label{eq3.60}
\dot{g}_{j\alpha}-\dot{g}_{j\beta}\leq 0,
\end{equation}

\noindent where,

\begin{align}
\begin{gathered}\label{eq3.61}
\dot{g}_{j\alpha}-\dot{g}_{j\beta}=\delta_{dis}\cdot (g_{j\alpha}-g_{j\beta})+(w_{j\beta}-w_{j\alpha})+q_{j\beta}^{+}\cdot (g_{j+1,\beta}-g_{j\beta})\\-q_{j\alpha}^{+}\cdot (g_{j+1,\alpha}-g_{j\alpha})+ q_{j\beta}^{-}\cdot (g_{j-1,\beta}-g_{j\beta}-f_{j}^{H})-q_{j\alpha}^{-}\cdot (g_{j-1,\alpha}-g_{j\alpha}-f_{j}^{H})\\+ \sum_{k\in B}\delta_{int}\cdot x_{jk}\cdot\bigl(q^{+k}_{j\beta}\cdot (g_{j+1,\beta}-g_{j\beta})+q^{-k}_{j\beta}\cdot (g_{j-1,\beta}-g_{j\beta}-f_j^H)\vspace{-3.4mm}\\ -q^{+k}_{j\alpha}\cdot (g_{j+1,\alpha}-g_{j\alpha})-q^{-k}_{j\alpha}\cdot (g_{j-1,\alpha}-g_{j\alpha}-f_j^H)\bigr) .
\end{gathered}
\end{align}

We substitute into \eqref{eq3.61} $g_{ij}$ and $x_{ij}$, taken from \eqref{eq3.41} and \eqref{eq3.24} respectively. Assuming, then, that $f_{j}^{H}$ is independent of $\delta$, and equating terms of similar order, in the main order $O(\delta^{-1})$ in small $\delta$, condition \eqref{eq3.60} is equivalent to (recall that we are in the ID$_{1}$ regime):

\begin{equation}\label{eq3.62}
q_{j\beta}^{+}\cdot (g_{j+1,\beta}^{0}-g_{j\beta}^{0})+q_{j\beta}^{-}\cdot (g_{j-1,\beta}^{0}-g_{j\beta}^{0})\leq q_{j\alpha}^{+}\cdot (g_{j+1,\alpha}^{0}-g_{j\alpha}^{0})+q_{j\alpha}^{-}\cdot (g_{j-1,\alpha}^{0}-g_{j\alpha}^{0}).
\end{equation}

Note that in the main order $O(\delta^{-1})$ in small $\delta$ (assuming that $f_{\alpha\beta}^{B}$ is also independent of $\delta$) for the specified boundary point of the cone, we get:

\begin{equation}\label{eq3.63}
g_{j\beta}^{0}=g_{j\alpha}^{0},
\end{equation}

\noindent while for all the other $i\in H$, such that $i\neq j$, will be:

\begin{equation}\label{eq3.64}
g_{i\beta}^{0}\leq g_{i\alpha}^{0}.
\end{equation}

Combining \eqref{eq3.63} and \eqref{eq3.64} we obviously get:

\begin{equation}\label{eq3.65}
g_{j\alpha}^{0}-g_{i\alpha}^{0}\leq g_{j\beta}^{0}-g_{i\beta}^{0},
\end{equation}

\noindent and rewriting \eqref{eq3.62} in the equivalent form:

\begin{equation}\label{eq3.66}
q_{j\alpha}^{+}\cdot (g_{j\alpha}^{0}-g_{j+1,\alpha}^{0})+q_{j\alpha}^{-}\cdot (g_{j\alpha}^{0}-g_{j-1,\alpha}^{0})\leq q_{j\beta}^{+}\cdot (g_{j\beta}^{0}-g_{j+1,\beta}^{0})+q_{j\beta}^{-}\cdot (g_{j\beta}^{0}-g_{j-1,\beta}^{0}),
\end{equation}

\noindent we check that condition \eqref{eq3.66} is satisfied when $q_{i\alpha}\leq q_{i\beta}$, $\forall i\in H$ (the first term is smaller or equal than the third term, the second term is smaller or equal than the fourth term).

But also for the case when $q_{i\beta}<q_{i\alpha}$, $\forall i\in H$, rewriting \eqref{eq3.65} in the equivalent form:

\begin{equation}\label{eq3.67}
g_{i\beta}^{0}-g_{j\beta}^{0}\leq g_{i\alpha}^{0}-g_{j\alpha}^{0},
\end{equation}

\noindent we check that \eqref{eq3.62} is satisfied (again the first term is smaller or equal than the third term, the second term is smaller or equal than the fourth term).

Thus we obtain the following main result:

\begin{theorem}\label{prop3.9}
Let Assumptions \ref{as3.2}, \ref{as3.3}, \ref{as3.4} and \ref{as3.5} hold. Assume additionally, $\forall \alpha,\beta\in B$, that:

\begin{equation}
q_{i\alpha}\leq q_{i\beta} \quad \text{or} \quad q_{i\beta}<q_{i\alpha}, \quad \forall i \in H.
\end{equation}

\noindent Then, for sufficiently small discounting $\delta_{dis}=\delta$, and relatively smaller binary interaction coefficient $\delta_{int}=\delta^{2}$, in the main order in small $\delta$, for any $T>t$, and for any initial occupation probability distribution $x(t)$, and any terminal payoffs such that:

\begin{equation*}
g_{i\beta}(T)-f_{\alpha\beta}^{B}\leq g_{i\alpha}(T),
\end{equation*} 

\noindent there exists a unique solution to the time-dependent discounted MFG consistency problem such that the control $u$ is stationary, and is given by $u_{i\alpha\to i\beta}=0$, $\forall i \in H$, $\forall\alpha, \beta \in B$, x(s) stays near the fixed point of Proposition \ref{prop3.8} as $s\to T$, and $g_{ij}(s)$ stays near the stationary solution of Proposition \ref{prop3.8} (almost for all time), for large $T-t$.
\end{theorem}

\section{Discussion}\label{sec3.5}
\noindent In this paper we formulate the interaction of a large number of small players under the pressure of a major player (principal), on $n$-dimensional arrays, having in mind the paradigm of individuals defending against a bio-terrorist; alternatively, the similar context of corrupted tax inspectors against a benevolent authority. The $n$-dimensional arrays dual structure naturally describes on the one hand the distribution of individuals among $m$ levels of `behaviour' (e.g. levels of defence) and on the other, their distribution according to a phenotypic characteristic among $n$ levels of `hierarchy' (e.g. levels of infection). Transitions on the first network structure are mainly subject to the individuals' control, while transitions on the second are mainly subject to the principal's pressure. Transitions on both structures may as well be an outcome of the individuals' binary interactions. Our model is a performance of a finite state non-linear Markov game combining mean-field, evolutionary, and pressure-resistance types of interaction. For our analysis, we consider the discounted mean-field game consistency problem. According to the general framework of mean-field games we analyse the forward-backward system of coupled equations, the kinetic equations governing the evolution of the individuals' distribution among the $n\times m$ states (forward equation), and the Hamilton-Jacobi-Bellman equation giving the individuals' optimal payoff (backward equation). We solve the stationary problem and we provide a link of the stationary solution to the time-dependent problem. For simplicity, we work in the asymptotic regimes of fast execution of personal decisions, weak binary interactions, and small payoff discounting in time. Considering a stationary control that is consistent with the assumption of fast execution of personal decisions, in the main order of small payoff discounting in time (or in the main order of weak binary interactions), we find that individuals will be uniformly distributed among the `behaviours' of the unique `hierarchy' level where the sum of rewards is maximised, and we obtain the optimal payoff as a function of these rewards. We show that there is a unique solution to the time-dependent problem, that is very close to the stationary solution. Our simplifications, while necessary for concrete calculations, represent the first step towards a more comprehensive treatment of the game that we have introduced and explicitly formulated.

\section*{Acknowledgements}
Stamatios Katsikas would like to acknowledge support from the Engineering and Physical Sciences Research Council (EPSRC). Vassilli Kolokoltsov would like to acknowledge support from the Russian Foundation for Basic Research (RFBR, grant No. 17-01-00069).

\section*{Conflicts of Interest}
The authors declare no conflict of interest.

\end{document}